\documentclass[11pt]{article}

\usepackage[a4paper,margin=1in]{geometry}
\usepackage{amsmath,amssymb,amsthm,mathtools}
\usepackage{tikz}
\usepackage{microtype}
\usepackage[colorlinks=true,linkcolor=blue,citecolor=blue,urlcolor=blue]{hyperref}

\newtheorem{theorem}{Theorem}

\newcommand{\XX}{\mathcal X}

\newcommand{\EE}{\mathcal E}

\newcommand{\mdim}{\operatorname{mdim}}

\title{A Variational Principle for Metric Mean Dimension via Lower Brin–Katok Local Entropy}
\author{Ruxi Shi\\[2mm]
\textit{Shanghai Center for Mathematical Science, Fudan University}\\
\texttt{ruxishi@fudan.edu.cn}}
\date{}

\begin{document}
\maketitle

\begin{abstract}
We prove a finite-scale comparison between lower Brin--Katok local entropy and Katok covering entropy.  Let $(\XX,d,T)$ be a compact metric topological dynamical system and let $\mu$ be ergodic.  Then, for every $\epsilon>0$ and every $\delta\in(0,1)$,
\[
        h^K_\mu(6\epsilon,\delta)\leq \underline h^{BK}_\mu(\epsilon).
\]
Combining this estimate with the usual Katok-type variational principle for metric mean dimension gives the corresponding variational principle with lower Brin--Katok local entropy.
\end{abstract}

\section{Introduction}

The purpose of this paper is to answer two finite-scale questions concerning lower Brin--Katok local entropy and metric mean dimension, which were formulated by Shi \cite{Shi21}.  Throughout, $\EE_T(\XX)$ denotes the set of ergodic $T$-invariant probability measures on $\XX$.  By a \emph{compact metric topological dynamical system} we mean a triple $(\XX,d,T)$ where $(\XX,d)$ is a compact metric space and $T:\XX\to\XX$ is a continuous map.

\subsection{Background}

Mean dimension was introduced by Gromov \cite{Gro99} and later developed by Lindenstrauss and Weiss \cite{LW00}, who also introduced metric mean dimension as a scale-sensitive analogue.  For a compact metric dynamical system $(\XX,d,T)$, the Bowen metric
\[
        d_n(x,y)=\max_{0\le j<n}d(T^j x,T^j y)
\]
induces, at each scale $\epsilon>0$, a finite-scale topological entropy
\[
        S(\epsilon)=\limsup_{n\to\infty}\frac1n\log \text{\rm Cov}(n,\epsilon),
\]
where $\text{\rm Cov}(n,\epsilon)$ denotes the least number of $d_n$-balls of radius $\epsilon$ needed to cover $\XX$.  The upper and lower metric mean dimensions are obtained by normalizing $S(\epsilon)$ and $\underline S(\epsilon)$ by $\log(1/\epsilon)$ and letting $\epsilon\downarrow0$.  This invariant is especially useful for systems whose topological entropy is infinite but whose finite-scale covering growth has a meaningful first-order rate.

The classical variational principle identifies topological entropy with the supremum of measure-theoretic entropy over invariant measures.  For metric mean dimension the analogous question is subtler for two reasons.  First, metric mean dimension depends on the chosen metric.  Second, the relevant measure-theoretic objects must retain the scale $\epsilon$; ordinary entropy $h_\mu(T)$ is usually infinite or too coarse to see the normalization by $\log(1/\epsilon)$.  The resulting theory therefore seeks finite-scale measure-theoretic entropies whose supremum over invariant measures recovers the finite-scale topological covering growth after passage to the metric mean dimension limit.

A major development was the rate-distortion variational principle of Lindenstrauss and Tsukamoto \cite{LT18}, which connected metric mean dimension with information-theoretic rate-distortion functions.  In that framework a measure is viewed as a stationary stochastic process and one asks how many bits per iterate are needed to describe a typical orbit with distortion at most $\epsilon$.  Lindenstrauss and Tsukamoto later developed a double variational principle, involving optimization over both metrics and measures, for mean dimension itself under the marker property \cite{LT19}.  Tsukamoto also introduced mean dimension with potential and proved a corresponding double variational principle, a mean-dimension analogue of the thermodynamic formalism for pressure \cite{Tsu20}.

Another line of progress replaces rate-distortion quantities by more geometric finite-scale entropies.  Velozo and Velozo obtained a variational principle for metric mean dimension using a Katok-type covering entropy related to measure-theoretic entropy at scale $\epsilon$ \cite{VV17}.  Gutman and \'{S}piewak proved that, in the Lindenstrauss--Tsukamoto variational principle, it is enough to take the supremum over ergodic measures; they also developed a formulation using partitions of decreasing diameter, related to dynamical R\'{e}nyi information dimension, and obtained a lower bound in terms of Brin--Katok local entropy \cite{GS21}.  Shi then established several variational principles for metric mean dimension in terms of Shapira entropy of finite open covers, Katok entropy, and Brin--Katok local entropy \cite{Shi21}.  In particular, Shi proved that for every fixed $\delta\in(0,1)$ the upper and lower metric mean dimensions can be recovered from
\[
        \sup_{\mu\in\EE_T(\XX)} h^K_\mu(\epsilon,\delta)
        \quad\text{and}\quad
        \sup_{\mu\in\EE_T(\XX)} h^{BK}_\mu(\epsilon)
\]
after the usual normalization by $\log(1/\epsilon)$.

These ideas have been further extended to countable amenable group actions \cite{CDZ22}, relative metric mean dimension with potential \cite{Wu23}, Bowen metric mean dimension for subsets \cite{Wang21}, and random dynamical systems \cite{WCY23}.  Together, these developments show that finite-scale local and covering entropies provide a robust measure-theoretic language for metric mean dimension.

\subsection{The problem}

The present paper concerns a lower-time version of this language.  The Brin--Katok local entropy formula expresses measure-theoretic entropy through the exponential decay of the measures of Bowen balls along typical orbits \cite{BK83}.  Katok's covering entropy gives a complementary global formulation: one counts the least number of Bowen balls needed to cover a set of prescribed positive $\mu$-measure \cite{Katok80}.  In the classical entropy limit these two viewpoints are compatible.  At a fixed scale, however, one must distinguish the upper local quantity
\[
        h^{BK}_\mu(x,\epsilon)=
        \limsup_{n\to\infty}-\frac1n\log\mu(B_n(x,\epsilon))
\]
from the lower local quantity
\[
        \underline h^{BK}_\mu(x,\epsilon)=
        \liminf_{n\to\infty}-\frac1n\log\mu(B_n(x,\epsilon)).
\]
For an ergodic measure these are almost everywhere constant, but they may behave very differently at finite scale.

The lower Brin--Katok quantity is not covered by the existing upper-time arguments.  The difficulty is that a lower local entropy bound gives information only along infinitely many good lengths, and these lengths may depend on the base point.  Thus one cannot directly compare, at a common time $n$, the measures of $n$-Bowen balls with the number of $n$-Bowen balls needed to cover a typical set.

A substantial body of work has shown that many local entropy structures---including rate-distortion functions, Katok entropy, Shapira entropy, and Brin--Katok local entropy---support variational principles for metric mean dimension.  The remaining missing piece is the lower Brin--Katok local entropy.  This gap led Shi to formulate two natural questions:

\begin{center}
\textit{%
Does the metric mean dimension variational principle remain true with lower Brin--Katok local entropy? Does lower Brin--Katok local entropy control lower Katok covering entropy at a comparable scale?}
\end{center}

\subsection{Main results}

We answer both questions affirmatively.  In fact, we prove the stronger finite-scale comparison stated in Theorem~\ref{thm:finite-scale} below, and derive the corresponding variational principle in Theorem~\ref{thm:vp}. 

\begin{theorem}\label{thm:finite-scale}
For every ergodic invariant measure $\mu$, every $\epsilon>0$, and every $\delta\in(0,1)$,
\[
        h^K_\mu(6\epsilon,\delta)\leq\underline h^{BK}_\mu(\epsilon).
\]
\end{theorem}

Since $\underline h^K_\mu(6\epsilon,\delta)\leq h^K_\mu(6\epsilon,\delta)$ by definition, Theorem~\ref{thm:finite-scale} immediately yields the desired lower-Katok inequality.  Combining the estimate with the Katok variational principle for metric mean dimension gives our second main result.

\begin{theorem}\label{thm:vp}
For every compact metric topological dynamical system $(\XX,d,T)$,
\[
        \overline{\mdim}_M(\XX,d,T)
        =
        \limsup_{\epsilon\downarrow0}
        \frac{1}{\log(1/\epsilon)}
        \sup_{\mu\in\EE_T(\XX)}\underline h^{BK}_\mu(\epsilon),
\]
and
\[
        \underline{\mdim}_M(\XX,d,T)
        =
        \liminf_{\epsilon\downarrow0}
        \frac{1}{\log(1/\epsilon)}
        \sup_{\mu\in\EE_T(\XX)}\underline h^{BK}_\mu(\epsilon).
\]
\end{theorem}

\begin{proof}
\textit{Lower bound.}  By Theorem~\ref{thm:finite-scale},
$\sup_{\mu\in\EE_T(\XX)}h^K_\mu(6\epsilon,\delta)\leq\sup_{\mu\in\EE_T(\XX)}\underline h^{BK}_\mu(\epsilon)$
for every $\delta\in(0,1)$.  The variational principle for metric mean dimension in term of Katok entropy \cite{Shi21} therefore yields
\[
        \overline{\mdim}_M(\XX,d,T)
        \leq\limsup_{\epsilon\downarrow0}\frac{\sup_{\mu\in\EE_T(\XX)}\underline h^{BK}_\mu(\epsilon)}{\log(1/\epsilon)}
        \quad\text{and}\quad
        \underline{\mdim}_M(\XX,d,T)
        \leq\liminf_{\epsilon\downarrow0}\frac{\sup_{\mu\in\EE_T(\XX)}\underline h^{BK}_\mu(\epsilon)}{\log(1/\epsilon)}.
\]
The constant factor $6$ in the scale is harmless because $\log(1/(6\epsilon))/\log(1/\epsilon)\to1$.

\textit{Upper bound.}  By definition $\underline h^{BK}_\mu(\epsilon)\leq h^{BK}_\mu(\epsilon)$, and the classical Brin--Katok inequality gives $h^{BK}_\mu(\epsilon)\leq S(\epsilon/2)$ for every $\mu\in\EE_T(\XX)$.  Taking supremum over $\mu$ and letting $\epsilon\downarrow0$ yields the reverse inequalities.
\end{proof}

Theorem~\ref{thm:vp} shows that lower Brin--Katok local entropy gives a complete finite-scale local entropy formulation of metric mean dimension, answering Shi's first question.  Theorem~\ref{thm:finite-scale} answers the second question in the stronger form with upper Katok entropy.

Since Theorem~\ref{thm:finite-scale} is available, the lower Brin--Katok formulation can be extended in the same directions as the previous local entropy frameworks.  In particular, the corresponding variational principles for countable amenable group actions \cite{CDZ22}, relative metric mean dimension with potential \cite{Wu23}, Bowen metric mean dimension for subsets \cite{Wang21}, and random dynamical systems \cite{WCY23} all carry over to the lower Brin--Katok setting with essentially the same proofs.
\subsection*{Idea of proof}

The proof of Theorem~\ref{thm:finite-scale} is based on a variable-length coding argument.  Given $a>\underline h^{BK}_\mu(\epsilon)$, let $E_m$ be the set of points whose $m$-Bowen ball has measure larger than $e^{-am}$.  For every large starting threshold $N$, the union $\bigcup_{m\ge N}E_m$ has full measure, hence a finite subunion $A=\bigcup_{m=N}^L E_m$ has measure close to one.  Birkhoff's theorem implies that most long orbit segments visit $A$ with high frequency.  Whenever the orbit is in $A_m\subset E_m$, we encode the next $m$ iterates by one of at most $e^{am}$ Bowen balls; the remaining exceptional times are encoded by one-step balls.  The number of exceptional symbols is small, and the number of good blocks is at most $n/N$.

A small but important point is the counting of variable-length decompositions.  The auxiliary integer $L$ may grow rapidly with $N$, so one should not pay a factor such as $(L+1)^{\#\text{blocks}}$.  Instead one counts the cut positions in $\{0,1,\dots,n\}$.  The resulting combinatorial overhead depends only on the density of blocks, namely approximately $1/N+2\alpha$, and vanishes after $N\to\infty$ and $\alpha\downarrow0$.  This gives the finite-scale estimate above.  The constant $6$ is only a consequence of elementary triangle-inequality losses in the proof and is not meant to be sharp.

\subsection*{Organization}

The paper is organized as follows.  Section~2 recalls the necessary notation and definitions.  Section~3 proves Theorem~\ref{thm:finite-scale} by the variable-length coding argument.
\section{Notation and preliminaries}

Throughout, $(\XX,d,T)$ is a compact metric topological dynamical system.  We write
\[
        d_n(x,y)=\max_{0\leq j<n} d(T^j x,T^j y),
        \qquad
        B_n(x,r)=\{y\in \XX: d_n(x,y)<r\},
\]
and denote by $\EE_T(\XX)$ the set of ergodic $T$-invariant probability measures on $\XX$.

\smallskip
\noindent\textit{Measure-theoretic entropies.}
For $\mu\in\EE_T(\XX)$, the pointwise lower and upper Brin--Katok local entropies at scale $r$ are
\[
        \underline h^{BK}_\mu(x,r)=\liminf_{n\to\infty}-\frac1n\log\mu(B_n(x,r))
        \quad\text{and}\quad
        h^{BK}_\mu(x,r)=\limsup_{n\to\infty}-\frac1n\log\mu(B_n(x,r)).
\]
For ergodic $\mu$, both quantities are almost everywhere constant; their a.e.\ constant values are denoted by $\underline h^{BK}_\mu(r)$ and $h^{BK}_\mu(r)$, respectively.
For $\delta\in(0,1)$, let $N^\delta_\mu(n,r)$ be the least number of $d_n$-balls of radius $r$ whose union has $\mu$-measure larger than $\delta$, and put
\[
        h^K_\mu(r,\delta)=\limsup_{n\to\infty}\frac1n\log N^\delta_\mu(n,r),
        \qquad
        \underline h^K_\mu(r,\delta)=\liminf_{n\to\infty}\frac1n\log N^\delta_\mu(n,r).
\]

\smallskip
\noindent\textit{Topological quantities.}
Let $\text{\rm Cov}(n,r)$ denote the least number of $d_n$-balls of radius $r$ needed to cover $\XX$, and define the finite-scale topological entropies
\[
        S(r)=\limsup_{n\to\infty}\frac1n\log \text{\rm Cov}(n,r)
        \quad\text{and}\quad
        \underline S(r)=\liminf_{n\to\infty}\frac1n\log \text{\rm Cov}(n,r).
\]
The \emph{upper} and \emph{lower metric mean dimensions} are then
\[
        \overline{\mdim}_M(\XX,d,T)
        =\limsup_{r\downarrow0}\frac{S(r)}{\log(1/r)}
        \qquad\text{and}\qquad
        \underline{\mdim}_M(\XX,d,T)
        =\liminf_{r\downarrow0}\frac{\underline S(r)}{\log(1/r)}.
\]

A subset $F\subset\XX$ is called \emph{$(n,\epsilon)$-separated} if distinct points of $F$ have $d_n$-distance strictly larger than $\epsilon$.  An $(n,\epsilon)$-separated set $F\subset A$ is \emph{maximal in $A$} if no strictly larger superset of $F$ inside $A$ is $(n,\epsilon)$-separated; equivalently, $A\subset\bigcup_{x\in F}B_n(x,\epsilon)$.

\section[pdfstring={Proof of Theorem 1}]{Proof of Theorem~\ref{thm:finite-scale}}

Fix an ergodic measure $\mu$, a number $\epsilon>0$, and $\delta\in(0,1)$.  If $\underline h^{BK}_\mu(\epsilon)=+\infty$, there is nothing to prove.  Hence assume $\underline h^{BK}_\mu(\epsilon)<+\infty$, and fix an arbitrary number $a>\underline h^{BK}_\mu(\epsilon)$.

\subsection[pdfstring={Step 1. Construction of the good set A}]{Step 1.  Construction of the good set $A$}

For each integer $m\geq 1$ define
\[
        E_m=\left\{x\in\XX:
        \mu(B_m(x,\epsilon))>e^{-am}\right\}.
\]
By definition of the lower Brin--Katok entropy, for $\mu$-almost every $x\in\XX$ we have
\[
        \liminf_{m\to\infty}-\frac1m\log\mu(B_m(x,\epsilon))
        =\underline h^{BK}_\mu(\epsilon)<a.
\]
This means that for $\mu$-a.e. $x$, there exist infinitely many integers $m$ such that
\[
        -\frac1m\log\mu(B_m(x,\epsilon))<a,
        \quad\text{equivalently}\quad
        \mu(B_m(x,\epsilon))>e^{-am}.
\]
Consequently, for $\mu$-a.e. $x$, there exists some $m\geq 1$ with $x\in E_m$.  Therefore
\[
        \mu\left(\bigcup_{m\geq 1}E_m\right)=1.
\]
Since the sets $\bigcup_{m\geq N}E_m$ decrease to the intersection $\bigcap_{N\geq 1}\bigcup_{m\geq N}E_m$, which still has full measure, we obtain
\[
        \mu\left(\bigcup_{m\geq N}E_m\right)=1
        \quad\text{for every integer }N\geq 1.
\]

Now fix an auxiliary parameter $0<\alpha<1/8$.  Because $\mu$ is a probability measure and the union $\bigcup_{m\geq N}E_m$ has full measure, for each integer $N>4$ we can choose an integer $L=L(N)\geq N$ such that the finite union
\[
        A:=\bigcup_{m=N}^{L}E_m
\]
satisfies
\[
        \mu(A)>1-\alpha.
\]
We then choose a measurable partition of $A$ into sets
\[
        A=\bigsqcup_{m=N}^{L}A_m,
        \qquad\text{with }A_m\subset E_m\text{ for each }m.
\]

\subsection[pdfstring={Step 2. Covering each Am by Bowen balls}]{Step 2.  Covering each $A_m$ by Bowen balls}

Fix $m\in\{N,\dots,L\}$.  Choose a maximal $(m,2\epsilon)$-separated subset $F_m\subset A_m$.

\textbf{Claim 1.}  The Bowen balls $\{B_m(x,\epsilon):x\in F_m\}$ are pairwise disjoint, and
\[
        |F_m|\leq e^{am}.
\]

\textit{Proof of Claim 1.} Obviously, the balls are pairwise disjoint. On the other hand, since each $x\in F_m\subset A_m\subset E_m$, by definition of $E_m$ we have $\mu(B_m(x,\epsilon))>e^{-am}$.  Summing over $F_m$ and using disjointness gives
$|F_m|\leq e^{am}$.
\qed

\textbf{Claim 2.}  Define
\[
        \mathcal C_m:=\{B_m(x,3\epsilon):x\in F_m\}.
\]
Then $\mathcal C_m$ covers $A_m$.

\textit{Proof of Claim 2.}  Let $y\in A_m$ be arbitrary.  If $y\in F_m$, then $y\in B_m(y,3\epsilon)\in\mathcal C_m$ trivially.  If $y \notin F_m$, then by maximality of $F_m$ the set $F_m\cup\{y\}$ is no longer $(m,2\epsilon)$-separated.  Hence there exists some $x\in F_m$ with $d_m(x,y)\leq 2\epsilon<3\epsilon$, which means $y\in B_m(x,3\epsilon)\in\mathcal C_m$.
\qed

From Claims 1 and 2, each $A_m$ is covered by at most $|\mathcal C_m|=|F_m|\leq e^{am}$ many $m$-Bowen balls of radius $3\epsilon$.

\subsection[pdfstring={Step 3. Potential candidate Zn for computing hK}]{Step 3.  Potential candidate $Z_n$ for computing $h^K_\mu(6\epsilon,\delta)$}

Since $\XX$ is compact, we can choose a finite cover $\mathcal Q$ of $\XX$ by ordinary $d$-balls of radius $3\epsilon$.  Let $Q:=|\mathcal Q|$ denote the cardinality of this cover.
Define
\[
        Z_n:=\biggl\{x\in\XX:
        \#\{0\leq t<n:T^t x \notin A\}\leq 2\alpha n
        \biggr\}.
\]
By Birkhoff’s ergodic theorem, it follows that $\mu(Z_n)\to 1$ as $n\to\infty$.  In particular, since $\delta\in(0,1)$ is fixed, we have
\[
        \mu(Z_n)>\delta
        \quad\text{for all sufficiently large }n.
\]
Our goal is now to cover $Z_n$ by $d_n$-balls of radius $6\epsilon$.

\subsection{Step 4.  Variable-length encoding of orbits}

Fix $x\in Z_n$ and an integer $n$ large enough that $\mu(Z_n)>\delta$.  We construct a code for the orbit segment $(x,Tx,\dots,T^{n-1}x)$ by the following algorithm, which proceeds by advancing a clock variable $t$ starting from $t=0$. We define a variable-length coding as follows:

\medskip
\noindent\textbf{Algorithm.}
\begin{enumerate}
\item[(0)] Initialize $t:=0$ and initialize an empty list of blocks.
\item[(1)] If $t\geq n$, stop.
\item[(2)] If there exists $m\in\{N,\dots,L\}$ such that $T^t x\in A_m$ and $t+m\leq n$, then:
        \begin{itemize}
        \item choose arbitrarily one ball $B_m(y,3\epsilon)\in\mathcal C_m$ that contains $T^t x$;
        \item append to the list the pair $(m,B_m(y,3\epsilon))$, called a \emph{good block of length $m$};
        \item update $t:=t+m$;
        \item return to step (1).
        \end{itemize}
\item[(3)] Otherwise (i.e., either $T^t x \notin A$, or $T^t x\in A_m$ but $t+m>n$):
        \begin{itemize}
        \item choose arbitrarily one ball $U\in\mathcal Q$ that contains $T^t x$;
        \item append to the list the pair $(1,U)$, called a \emph{bad block of length $1$};
        \item update $t:=t+1$;
        \item return to step (1).
        \end{itemize}
\end{enumerate}

Since $t$ strictly increases at each iteration and we stop when $t\geq n$, the algorithm terminates in finitely many steps.  The output is a finite sequence (code) consisting of blocks, each block being either {\it good} (with some length $m\in[N,L]$) or {\it bad} (with length $1$).

Figure~\ref{fig:coding} illustrates the variable-length coding procedure.

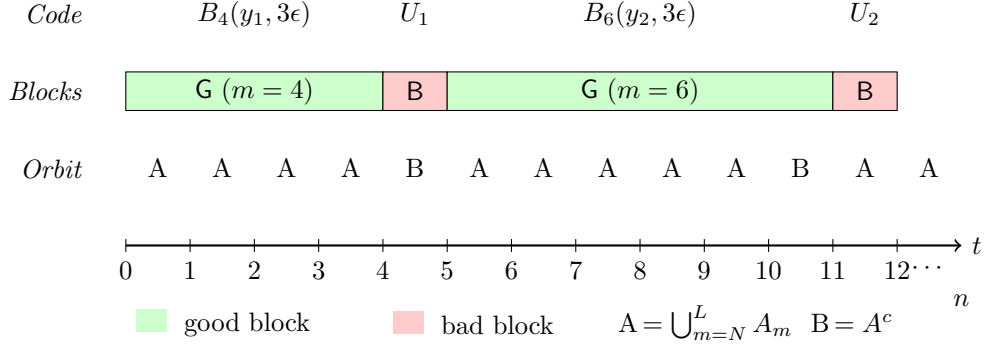
\begin{figure}[ht]
\centering
\begin{tikzpicture}[scale=0.85, every node/.style={font=\small}]
\draw[->, thick] (0,0) -- (13,0) node[right] {$t$};
\foreach \x in {0,1,2,3,4,5,6,7,8,9,10,11,12}
    \draw (\x,0.1) -- (\x,-0.1) node[below] {\x};

\node[below] at (12.5,-0.1) {$\cdots$};

\node[below] at (13,-0.6) {$n$};


\node[left] at (-0.5,1.2) {\textit{Orbit}};
\foreach \x/\c in {0/A,1/A,2/A,3/A,4/B,5/A,6/A,7/A,8/A,9/A,10/B,11/A,12/A} {
        \node at (\x+0.5,1.2) {\c};
    }


\node[left] at (-0.5,2.4) {\textit{Blocks}};
\fill[green!20] (0,2.1) rectangle (4,2.7);
\draw (0,2.1) rectangle (4,2.7);

\node at (2,2.4) {\textsf{G} ($m=4$)};
\fill[red!20] (4,2.1) rectangle (5,2.7);
\draw (4,2.1) rectangle (5,2.7);

\node at (4.5,2.4) {\textsf{B}};
\fill[green!20] (5,2.1) rectangle (11,2.7);
\draw (5,2.1) rectangle (11,2.7);

\node at (8,2.4) {\textsf{G} ($m=6$)};
\fill[red!20] (11,2.1) rectangle (12,2.7);
\draw (11,2.1) rectangle (12,2.7);

\node at (11.5,2.4) {\textsf{B}};


\node[left] at (-0.5,3.6) {\textit{Code}};

\node at (2,3.6) {$B_4(y_1,3\epsilon)$};

\node at (4.5,3.6) {$U_1$};

\node at (8,3.6) {$B_6(y_2,3\epsilon)$};

\node at (11.5,3.6) {$U_2$};


\node[right] at (0,-1.2) {\tikz\fill[green!20] (0,0) rectangle (0.4,0.3);~~good block};

\node[right] at (4,-1.2) {\tikz\fill[red!20] (0,0) rectangle (0.4,0.3);~~bad block};

\node[right] at (7.5,-1.2) {A$\,=\bigcup_{m=N}^L A_m$};

\node[right] at (10.5,-1.2) {B$\,=A^c$};
\end{tikzpicture}
\caption{Variable-length coding of an orbit segment.  The first row records whether $T^t x$ lies in the good set $A$ (labelled~A) or in its complement (labelled~B).  The second row shows the resulting block decomposition: green rectangles are good blocks, red rectangles are bad blocks.  The third row records the covering ball chosen for each block.}
\label{fig:coding}
\end{figure}

\subsection{Step 5.  Counting bad blocks and total blocks}

\textbf{Claim 3 (number of bad blocks).}  For any $x\in Z_n$, the number of bad blocks produced by the algorithm is at most $2\alpha n+L$.

\textit{Proof of Claim 3.}  A bad block is created in step (3) of the algorithm.  This occurs in exactly two mutually exclusive situations:
\begin{itemize}
\item[(i)] $T^t x \notin A$.  For $x\in Z_n$, the number of times $t\in\{0,1,\dots,n-1\}$ with $T^t x \notin A$ is by definition at most $2\alpha n$.
\item[(ii)] $T^t x\in A_m$ for some $m\in\{N,\dots,L\}$ but $t+m>n$.  In this case $t>n-m\geq n-L$.  Since $t\in\{0,1,\dots,n-1\}$, the number of such $t$ is at most $L$.
\end{itemize}
Summing the two contributions yields at most $2\alpha n+L$ bad blocks.
\qed

\textbf{Claim 4 (number of good blocks).}  The number of good blocks is at most $\frac{n}{N}$.

\textit{Proof of Claim 4.}  Each good block has length at least $N$ by construction, and the sum of the lengths of all blocks equals exactly $n$ (the algorithm covers precisely the time interval $[0,n)$).  If there are $g$ good blocks with lengths $m_1,\dots,m_g\geq N$, then
\[
        n\geq\sum_{i=1}^g m_i\geq gN,
\]
so $g\leq n/N$.
\qed

Consequently, the total number of blocks is at most $n/N+2\alpha n+L+1$.

\subsection{Step 6.  Counting the number of possible codes}

We now estimate how many distinct codes can arise from all points $x\in Z_n$.  A code consists of three pieces of information:
\begin{enumerate}
\item[(a)] the sequence of block lengths,
\item[(b)] for each block, whether it is good or bad,
\item[(c)] for each block, which covering ball was chosen.
\end{enumerate}
We estimate the number of possibilities for each component.

\medskip
\noindent\textbf{Counting (c): the covering choices.}

For a good block of length $m$, the chosen ball comes from $\mathcal C_m$, which has cardinality at most $e^{am}$.  Hence the total number of possibilities for all good blocks is at most
\[
       |\{\text{Choice of good blocks}\}|\le \prod_{\text{good blocks of length }m}e^{am}\le  e^{an}.
\]
The last inequality follows from the fact that the sum of good block lengths is at most $n$. 

For each bad block, the chosen ball comes from $\mathcal Q$, which has cardinality $Q$.  Since there are at most $2\alpha n+L$ bad blocks, the bad-block covering choices contribute at most
\[
        |\{\text{Choice of bad blocks}\}|\le Q^{2\alpha n+L}.
\]

\medskip
\noindent\textbf{Counting (a)--(b): the block decompositions.}

This is the crucial combinatorial step.  Instead of trying to count sequences of lengths $m\in[N,L]$ directly (which would incur a factor involving $L$), we count the \emph{cut positions}---the times at which one block ends and the next begins.

Let $k$ denote the total number of blocks.  By Step 5,
\[
        k\leq\frac{n}{N}+2\alpha n+L+1.
\]
Define
\[
        \beta_n:=\frac{1}{N}+2\alpha+\frac{L+1}{n}.
\]

The cut positions are integers $0=t_0<t_1<\dots<t_k=n$ such that block $i$ occupies the time interval $[t_{i-1},t_i)$.  Specifying the set of cut positions $\{t_1,\dots,t_{k-1}\}$ is equivalent to specifying a subset of $\{1,2,\dots,n-1\}$ of cardinality $k-1$.  For fixed $k$, the number of such subsets is $\binom{n-1}{k-1}\leq\binom{n}{k}$.

Additionally, each block must be labeled as good or bad.  For $k$ blocks there are $2^k$ possible labelings.

Therefore, for a fixed $k$, the number of block decompositions with exactly $k$ blocks is at most
\[
        \binom{n}{k}\cdot 2^k.
\]
Summing over all integers $k$ with $0\leq k\leq\lfloor\beta_n n\rfloor$, the total number of block decompositions is bounded by
\[
        \sum_{k=0}^{\lfloor\beta_n n\rfloor}\binom{n}{k}2^{k}.
\]

\medskip
\noindent\textbf{Claim 5 (asymptotics of the combinatorial factor).}  Let
\[
        \eta(u):=-u\log u-(1-u)\log(1-u)+u\log 2
        \quad\text{for }0<u<\tfrac12.
\]
Then $\eta(u)\to 0$ as $u\downarrow 0$, and for fixed $N,\alpha,L$ we have
\[
        \limsup_{n\to\infty}\frac1n
        \log\left(\sum_{k\leq\beta_n n}\binom{n}{k}2^{k}\right)
        \leq\eta\left(\frac1N+2\alpha\right).
\]

\textit{Proof of Claim 5.}  Since $\beta_n\to\frac1N+2\alpha$ as $n\to\infty$ and $\frac1N+2\alpha<\frac14+\frac14=\frac12$ (because $N>4$ and $\alpha<1/8$ by assumption), we may assume $n$ is large enough that $\beta_n<1/2$.

For any $u\in(0,1/2)$, Stirling's approximation gives the standard binomial bound
\[
        \sum_{k\leq un}\binom{n}{k}\leq\exp(n\,H(u)+o(n)),
\]
where $H(u)=-u\log u-(1-u)\log(1-u)$ is the binary entropy function.  Multiplying by $2^{un}=\exp((u\log 2)\cdot n+o(n))$ gives
\[
        \sum_{k\leq un}\binom{n}{k}2^{k}
        \leq\exp\bigl(n(H(u)+u\log 2)+o(n)\bigr)
        =\exp(n\,\eta(u)+o(n)).
\]
Taking $u=\beta_n$ and letting $n\to\infty$ yields the claimed bound.
\qed

The crucial observation is that this combinatorial overhead depends only on $\frac1N+2\alpha$, and is completely independent of $L$.  Since $L$ may grow very rapidly with $N$, any estimate involving $L$ directly (such as $(L+1)^{\#\text{blocks}}$) would not vanish in the limit.

\subsection[pdfstring={Step 7. Same code implies dn-proximity}]{Step 7.  Same code implies $d_n$-proximity}

\textbf{Claim 6.}  If two points $x,y\in Z_n$ produce the same code, then $d_n(x,y)<6\epsilon$.

\textit{Proof of Claim 6.}  Suppose $x$ and $y$ have the same code.  Then their orbit segments are decomposed into exactly the same sequence of blocks.  Consider a single block starting at time $t$.

\textit{Case 1: good block of length $m$.}  By definition of the code, both $T^t x$ and $T^t y$ belong to the same ball $B_m(z,3\epsilon)\in\mathcal C_m$ for some $z\in F_m$.  Hence
\[
        d_m(T^t x,T^t y)
        \leq d_m(T^t x,z)+d_m(z,T^t y)
        <3\epsilon+3\epsilon=6\epsilon.
\]
By definition of $d_m$, this means that for every $j\in\{0,1,\dots,m-1\}$,
\[
        d(T^{t+j}x,T^{t+j}y)<6\epsilon.
\]

\textit{Case 2: bad block of length $1$.}  By definition, both $T^t x$ and $T^t y$ belong to the same ball $U\in\mathcal Q$, which has $d$-radius $3\epsilon$.  Hence
\[
        d(T^t x,T^t y)<6\epsilon.
\]

Since every time $s\in\{0,1,\dots,n-1\}$ belongs to exactly one block, and in either case we have $d(T^s x,T^s y)<6\epsilon$, we conclude
\[
        d_n(x,y)=\max_{0\leq s<n}d(T^s x,T^s y)<6\epsilon.
\]
This completes the proof of Claim 6.
\qed

It follows from Claim 6 that each nonempty set of points in $Z_n$ sharing the same code is contained in a single $d_n$-ball of radius $6\epsilon$.

\subsection{Step 8.  Assembling the estimate and taking limits}

Combining the counting estimates from Step 6 with Claim 6, we deduce that $Z_n$ can be covered by at most
\[
        e^{an}\cdot Q^{2\alpha n+L}\cdot
        \sum_{k\leq\beta_n n}\binom{n}{k}2^{k}
\]
many $d_n$-balls of radius $6\epsilon$.  Since $\mu(Z_n)>\delta$ for all sufficiently large $n$, the definition of $N^\delta_\mu(n,6\epsilon)$ gives
\[
        N^\delta_\mu(n,6\epsilon)
        \leq
        \exp(an)\cdot Q^{2\alpha n+L}
        \cdot\sum_{k\leq\beta_n n}\binom{n}{k}2^{k}
\]
for all large $n$.

Taking logarithms, dividing by $n$, and letting $n\to\infty$, we obtain
\[
        \limsup_{n\to\infty}\frac1n\log N^\delta_\mu(n,6\epsilon)
        \leq a+2\alpha\log Q+\eta\left(\frac1N+2\alpha\right).
\]
By definition of Katok entropy, the left-hand side is exactly $h^K_\mu(6\epsilon,\delta)$.  Hence
\[
        h^K_\mu(6\epsilon,\delta)
        \leq a+2\alpha\log Q+\eta\left(\frac1N+2\alpha\right).
\]

We now perform the limiting procedure in three stages:
\begin{enumerate}
\item[(i)] First let $N\to\infty$.  This causes $\frac1N\to 0$, so the argument of $\eta$ tends to $2\alpha$.
\item[(ii)] Then let $\alpha\downarrow 0$.  Since $\eta(u)\to 0$ as $u\downarrow 0$ and $Q$ is independent of $\alpha$, we obtain
        \[
                h^K_\mu(6\epsilon,\delta)\leq a.
        \]
\item[(iii)] Finally, recall that $a>\underline h^{BK}_\mu(\epsilon)$ was arbitrary.  Letting $a\downarrow\underline h^{BK}_\mu(\epsilon)$ yields
        \[
                h^K_\mu(6\epsilon,\delta)\leq\underline h^{BK}_\mu(\epsilon).
        \]
\end{enumerate}
This proves Theorem \ref{thm:finite-scale}.

\section*{Acknowledgements}

The author thanks Weisheng Wu for bringing the attention back to the questions raised in \cite{Shi21} during his visit to Fudan University in March, which led the author to reconsider the problem.

\end{document}